\documentclass[12pt]{article}%
\usepackage[intlimits]{amsmath}
\usepackage{amssymb}
\usepackage{latexsym}
\usepackage[T1]{fontenc}
\usepackage[sc]{mathpazo}
\usepackage{color}
\usepackage[colorlinks]{hyperref}
\usepackage{amsfonts}
\usepackage{graphicx}%
\definecolor {refcol}{RGB}{40,0,255}
\hypersetup{colorlinks=true,allcolors=refcol}
\makeatletter
\newfont{\footsc}{cmcsc10 at 8truept}
\newfont{\footbf}{cmbx10 at 8truept}
\newfont{\footrm}{cmr10 at 10truept}
\newtheorem{theorem}{Theorem}

\newenvironment{proof}[1][Proof]{\noindent{\textbf {#1}  }}  {\hfill$\Box$\bigskip}
\begin{document}

\title{\textbf{Max }$k$\textbf{-cut and the smallest eigenvalue}}
\author{V. Nikiforov\thanks{Department of Mathematical Sciences, University of
Memphis, Memphis TN 38152, USA}}
\date{}
\maketitle

\begin{abstract}
Let $G$ be a graph of order $n$ and size $m$, and let $\mathrm{mc}_{k}\left(
G\right)  $ be the maximum size of a $k$-cut of $G.$ It is shown that
\[
\mathrm{mc}_{k}\left(  G\right)  \leq\frac{k-1}{k}\left(  m-\frac{\mu_{\min
}\left(  G\right)  n}{2}\right)  ,
\]
where $\mu_{\min}\left(  G\right)  $ is the smallest eigenvalue of the
adjacency matrix of $G.$

An infinite class of graphs forcing equality in this bound is
constructed.\textit{\medskip}

\textbf{Keywords: }max $k$-cut\textit{; chromatic number; largest eigenvalues;
largest Laplacian eigenvalue; smallest adjacency eigenvalue.\medskip}


\end{abstract}

\section{Introduction and main results}

The \emph{maximum }$k$\emph{-cut} of $G,$ denoted by $\mathrm{mc}_{k}\left(
G\right)  $, is the maximum number of edges in a $k$-partite subgraph of $G.$
This note provides an upper bound on $\mathrm{mc}_{k}\left(  G\right)  $ based
on $\mu_{\min}\left(  G\right)  $ -- the smallest eigenvalue of the adjacency
matrix of $G$.\ 

In \cite{MoPo90} Mohar and Poljak gave the celebrated bound $\mathrm{mc}%
_{2}\left(  G\right)  \leq$ $\lambda\left(  G\right)  n/4,$ where
$\lambda\left(  G\right)  $ is the maximum eigenvalue of the Laplacian matrix
of $G$. However, one may question how fit $\lambda\left(  G\right)  $ is for
such a bound on $\mathrm{mc}_{k}\left(  G\right)  $, since $\mathrm{mc}%
_{k}\left(  G\right)  $ is a Lipschitz function in the number of edges $m,$
whereas $\lambda\left(  G\right)  $ may be quite volatile in $m/n.$ Indeed,
raising the degree of a single vertex of maximum degree $\Delta\left(
G\right)  $ in $G$ can raise $\lambda\left(  G\right)  $ accordingly, due to
the inequality $\lambda\left(  G\right)  >\Delta\left(  G\right)  .$ In
contrast, $\mu_{\min}\left(  G\right)  $ depends more robustly on $m/n,$ and
hence may be a better choice than $\lambda\left(  G\right)  $ for upper bounds
on $\mathrm{mc}_{2}\left(  G\right)  $. In \cite{Tre12}, Trevisan came to
grips with similar problems, but the emphasis of his work is on algorithms and
no bound was produced in closed form. Thus, we propose the following theorem:

\begin{theorem}
\label{mth}If $G$ is a graph with $n$ vertices and $m$ edges, then
\begin{equation}
\mathrm{mc}_{k}\left(  G\right)  \leq\frac{k-1}{k}\left(  m-\frac{\mu_{\min
}\left(  G\right)  n}{2}\right)  . \label{mcin}%
\end{equation}

\end{theorem}

\begin{proof}
Let $G$ be as required and suppose that its vertex set is $\left[  n\right]
:=\left\{  1,\ldots,n\right\}  .$ Let $H$ be a $k$-partite subgraph of $G$
with $\mathrm{mc}_{k}\left(  G\right)  $ edges, and let $\left[  n\right]
=V_{1}\cup\cdots\cup V_{k}$ be the partition of the vertices of $H$ into $k$
edgeless sets. The idea of the proof is to use Rayleigh's principle to
construct $k$ upper bounds on $\mu_{\min}\left(  G\right)  ,$ and then take
their average as an upper bound on $\mu_{\min}\left(  G\right)  $.

For each $i\in\left[  k\right]  ,$ define a vector $\mathbf{y}^{(i)}%
:=(y_{1}^{(i)},\ldots,y_{n}^{(i)})$ as%
\[
y_{j}^{(i)}:=\left\{
\begin{array}
[c]{ll}%
-k+1, & \text{if }j\in V_{i},\medskip\\
1, & \text{if }j\in\left[  n\right]  \backslash V_{i}.
\end{array}
\right.
\]
Write $\left\langle \mathbf{u},\mathbf{v}\right\rangle $ for the inner product
of the vectors $\mathbf{u}$ and $\mathbf{v,}$ and note that for each
$i\in\left[  k\right]  ,$ Rayleigh's principle implies that
\[
\mu_{\min}\left(  G\right)  \left\Vert \mathbf{y}^{(i)}\right\Vert ^{2}%
\leq\left\langle A\mathbf{y}^{(i)},\mathbf{y}^{(i)}\right\rangle .
\]
Hence, summing these inequalities for all $i\in\left[  k\right]  ,$ we get%
\begin{equation}
\mu_{\min}\left(  G\right)  \sum\limits_{i\in\left[  k\right]  }\left\Vert
\mathbf{y}^{(i)}\right\Vert ^{2}\leq\sum\limits_{i\in\left[  k\right]
}\left\langle A\mathbf{y}^{(i)},\mathbf{y}^{(i)}\right\rangle .\label{i1}%
\end{equation}
On the one hand, for $\sum_{i\in\left[  k\right]  }\left\Vert \mathbf{y}%
^{(i)}\right\Vert ^{2}$ we have
\begin{equation}
\sum\limits_{i\in\left[  k\right]  }\left\Vert \mathbf{y}^{(i)}\right\Vert
^{2}=\sum\limits_{i\in\left[  k\right]  }\left(  \left(  k-1\right)
^{2}\left\vert V_{i}\right\vert +n-\left\vert V_{i}\right\vert \right)
=\left(  k^{2}-k\right)  n.\label{i2}%
\end{equation}
On the other hand, writing $e\left(  X\right)  $ for the number of edges
induced by a set $X$ and $e\left(  X,Y\right)  $ for the number of cross-edges
between the sets $X$ and $Y,$ for every $i\in\left[  k\right]  ,$ we see that
\[
\left\langle A\mathbf{y}^{(i)},\mathbf{y}^{(i)}\right\rangle =2\left(
k-1\right)  ^{2}e\left(  V_{i}\right)  +\sum\limits_{j\in\left[  k\right]
\backslash\left\{  i\right\}  }2e\left(  V_{j}\right)  -\sum\limits_{j\in
\left[  k\right]  \backslash\left\{  i\right\}  }2\left(  k-1\right)  e\left(
V_{i},V_{j}\right)  +\sum\limits_{j,l\in\left[  k\right]  \backslash\left\{
i\right\}  ,j\neq l}2e\left(  V_{l},V_{j}\right)  .
\]
Summing these inequalities for all $i\in\left[  k\right]  ,$ we get four terms
in the right side:
\begin{align*}
\sum\limits_{i\in\left[  k\right]  }2\left(  k-1\right)  ^{2}e\left(
V_{i}\right)   &  =2\left(  k-1\right)  ^{2}\left(  m-\mathrm{mc}_{k}\right)
,\\
\sum\limits_{i\in\left[  k\right]  }\sum\limits_{j\in\left[  k\right]
\backslash\left\{  i\right\}  }2e\left(  V_{j}\right)   &  =2\left(
k-1\right)  \left(  m-\mathrm{mc}_{k}\right)  ,\\
-\sum\limits_{i\in\left[  k\right]  }\sum\limits_{j\in\left[  k\right]
\backslash\left\{  i\right\}  }2\left(  k-1\right)  e\left(  V_{i}%
,V_{j}\right)   &  =-4\left(  k-1\right)  \mathrm{mc}_{k},\\
\sum\limits_{i\in\left[  k\right]  }\text{ }\sum\limits_{j,l\in\left[
k\right]  \backslash\left\{  i\right\}  ,j\neq l}2e\left(  V_{l},V_{j}\right)
&  =2\left(  k-2\right)  \mathrm{mc}_{k}.
\end{align*}
Hence, for $\sum_{i\in\left[  k\right]  }\left\langle A\mathbf{y}%
^{(i)},\mathbf{y}^{(i)}\right\rangle $ we obtain
\begin{align*}
\sum\limits_{i\in\left[  k\right]  }\left\langle A\mathbf{y}^{(i)}%
,\mathbf{y}^{(i)}\right\rangle  &  =2\left(  k-1\right)  ^{2}\left(
m-\mathrm{mc}_{k}\right)  +2\left(  k-1\right)  \left(  m-\mathrm{mc}%
_{k}\right)  -4\left(  k-1\right)  \mathrm{mc}_{k}+2\left(  k-2\right)
\mathrm{mc}_{k}\\
&  =2k\left(  k-1\right)  \left(  m-\mathrm{mc}_{k}\right)  -2k\mathrm{mc}%
_{k}=2k\left(  k-1\right)  \left(  m-\frac{k}{k-1}\mathrm{mc}_{k}\right)  .
\end{align*}
Finally, combining the last equality with (\ref{i1}) and (\ref{i2}), we get%
\[
\frac{\mu_{\min}\left(  G\right)  n}{2}\leq m-\frac{k}{k-1}\mathrm{mc}_{k},
\]
completing the proof of (\ref{mcin}).
\end{proof}

Note that the above proof also applies to weighted graphs, i.e., graphs whose
edges have been assigned positive real numbers. For $k=2,$ inequality
(\ref{mcin}) can be obtained from Lemma 1 of Delorme and Poljak \cite{DePo93}
by letting $u=\left[  2m/n-d_{i}\right]  ,$ where $d_{1},\ldots,d_{n}$ are the
degrees of $G$\footnote{Lov\'{a}sz stated this fact in Proposition 6.4.4
\cite{Lov03} without details, and Trevisan missed his point in the footnote on
p. 1772 of \cite{Tre12}.}. Likewise, (\ref{mcin}) can also be obtained by
letting $d=\left[  2m/n-d_{i}\right]  $ in equation (9) of the paper  of van
Dam and Sotirov \cite{DaSo16}.$\medskip$

Let us note that equality may hold in (\ref{mcin}) for numerous graphs, both
regular and irregular. Our next goal is to exhibit an infinite class of such
graphs, for which we need some preparation.\medskip

Suppose that $r\geq k\geq2$ and write $t_{k}\left(  n\right)  $ for the
maximum number of edges in a $k$-partite graph of order $n.$ The numbers
$t_{k}\left(  n\right)  $ are called Tur\'{a}n numbers, and it is known that
\[
t_{k}\left(  n\right)  =\frac{k-1}{2k}\left(  n^{2}-s^{2}\right)  +\binom
{s}{2},
\]
where $s$ is the remainder $n$ $\operatorname{mod}$ $k.$ It is not hard to see
that%
\[
\frac{k-1}{2k}n^{2}-\frac{k}{8}\leq t_{k}\left(  n\right)  \leq\frac{k-1}%
{2k}n^{2}.
\]
Equality on the right holds if and only if $k$ divides $n.$ Equality on the
left holds if and only if $k$ is even and $n=k/2$ $\operatorname{mod}$
$k.\medskip$

The sum of all weights of a weighted graph is called its \emph{total weight,
}and the maximum\emph{ }$k$-cut of a weighted graph is the maximum total
weight of its $k$-partite subgraphs$.$ Next, we give a lower bound on
$\mathrm{mc}_{k}\left(  G\right)  $ that may well be known.

\begin{theorem}
\label{th2}Let $r\geq k\geq2.$ If $G$ is a weighted $r$-partite graph with
total weight $m$, then
\[
\mathrm{mc}_{k}\left(  G\right)  \geq\frac{t_{k}\left(  r\right)  }{\binom
{r}{2}}m.
\]

\end{theorem}

\begin{proof}
Let $K$ be the weighted complete graph of order $r,$ whose vertices are the
vertex classes of $G,$ and the edge weights are the sums of the weights of all
edges across the corresponding classes. Clearly the total weight of $K$ is
$m.$ Define a random variable $\mathbf{X}_{k}\left(  K\right)  $ equal to the
total weight of a randomly chosen $k$-partite subgraph of $K$ with
$t_{k}\left(  r\right)  $ edges. Let $M$ be the number of all such subgraphs
of $K$. By symmetry, each edge of $K$ belongs to the same number of such
subgraphs, which obviously is
\[
\frac{t_{k}\left(  r\right)  }{\binom{r}{2}}M.
\]
Therefore,
\[
\mathbb{E}\left(  \mathbf{X}_{k}\left(  K\right)  \right)  =\frac{1}{M}%
\cdot\frac{t_{k}\left(  r\right)  }{\binom{r}{2}}Mm=\frac{t_{k}\left(
r\right)  }{\binom{r}{2}}m.
\]
Thus, there is a $k$-partite subgraph of $G$ of total weight at least
$t_{k}\left(  r\right)  m/\binom{r}{2},$ as claimed.
\end{proof}

Note that Theorem \ref{th2} is an improvement over the straightforward lower
bound $\mathrm{mc}_{k}\left(  G\right)  \geq\left(  1-1/k\right)  m.$

We are now ready to describe a class of regular graphs that force equality in
(\ref{mcin}).\medskip

Let $\chi\geq k\geq2$ and suppose that $k$ divides $\chi.$ Take a $t$-regular
graph $H$ of order $n$ satisfying $\omega\left(  H\right)  \geq\chi$ and
\[
\left\vert \mu_{\min}\left(  H\right)  \right\vert <\frac{t}{\chi-1}.
\]

Let $J_{\chi}$ be the $\chi\times\chi$ matrix of all-ones and $I_{\chi}$ be
the identity matrix of order $\chi.$ Write $A\left(  H\right)  $ for \ the
adjacency matrix of $H.$ The Kronecker product $B:=\left(  J_{\chi}-I_{\chi
}\right)  \otimes A\left(  H\right)  $ is a symmetric $\left(  0,1\right)
$-matrix with zero diagonal. Let $G$ be the graph with adjacency matrix $B.$
Clearly, $G$ is $\left(  \chi-1\right)  t$-regular graph of order $\chi n.$
Also,
\[
\mu_{\min}\left(  G\right)  =\min\left\{  -t,\left(  \chi-1\right)  \mu_{\min
}\left(  H\right)  \right\}  =-t.
\]

Using the fact that $\omega\left(  H\right)  \geq\chi,$ one can show that
$\omega\left(  G\right)  =\chi,$ which obviously implies that $\chi\left(
G\right)  =\chi$ as well. Since $k$ divides $\chi$, we have $t_{k}\left(
\chi\right)  =\frac{k-1}{2k}\chi^{2}$ and Theorem \ref{th2} implies that
\begin{align*}
\mathrm{mc}_{k}\left(  G\right)   &  \geq\frac{\frac{k-1}{2k}\chi^{2}}%
{\frac{\chi\left(  \chi-1\right)  }{2}}e\left(  G\right)  =\frac{k-1}{k}%
\frac{\chi}{\chi-1}e\left(  G\right)  =\frac{k-1}{k}e\left(  G\right)
+\frac{k-1}{k}\frac{t\left(  \chi-1\right)  \chi n}{2\left(  \chi-1\right)
}\\
&  =\frac{k-1}{k}\left(  e\left(  G\right)  +\frac{tv\left(  G\right)  }%
{2}\right)  =\frac{k-1}{k}\left(  e\left(  G\right)  -\frac{\mu_{\min}\left(
G\right)  v\left(  G\right)  }{2}\right)  .
\end{align*}
Hence the graph $G$ forces equality in (\ref{mcin}).\bigskip

To conclude, we show that two results of the recent paper\cite{DaSo16} are
simple consequences of a result proved in \cite{Nik07}. In \cite{DaSo16}, van
Dam and Sotirov showed that
\begin{equation}
\mathrm{mc}_{k}\left(  G\right)  \leq\frac{n\left(  k-1\right)  }{2k}%
\lambda\left(  G\right)  ,\label{DS}%
\end{equation}
where $\lambda\left(  G\right)  $ is the maximum eigenvalue of the Laplacian
matrix of $G.$ However, (\ref{DS}) follows immediately from an inequality in
\cite{Nik07} that reads as:\medskip

\emph{If }$H$\emph{\ is a }$k$\emph{-partite graph and }$\mu\left(  H\right)
$\emph{\ is the maximum eigenvalue of its adjacency matrix, then}%
\begin{equation}
\mu\left(  H\right)  \leq\frac{k-1}{k}\lambda\left(  H\right)  . \label{min}%
\end{equation}

Indeed, if $H$ is a $k$-partite subgraph of $G$ with $\mathrm{mc}_{k}\left(
G\right)  $ edges, then
\[
\frac{2\mathrm{mc}_{k}\left(  G\right)  }{n}\leq\mu\left(  H\right)  \leq
\frac{k-1}{k}\lambda\left(  H\right)  \leq\frac{k-1}{k}\lambda\left(
G\right)  ,
\]
and inequality (\ref{DS}) follows. Note that for regular graphs (\ref{DS}) and
(\ref{mcin}) are equivalent, but they are incomparable in general.\medskip

Further, van Dam and Sotirov show that if $G$ has $m$ edges, then its
chromatic number $\chi\left(  G\right)  $ satisfies:%
\begin{equation}
\chi\left(  G\right)  \geq1+\frac{2m}{n\lambda\left(  G\right)  -2m}%
.\label{DS1}%
\end{equation}

However, this inequality is also a simple consequence of (\ref{min}). Indeed,
rewriting (\ref{min}) as
\[
\chi\left(  G\right)  \geq1+\frac{\mu\left(  G\right)  }{\lambda\left(
G\right)  -\mu\left(  G\right)  },
\]
inequality (\ref{DS1}) follows as
\[
\chi\left(  G\right)  \geq1+\frac{\mu\left(  G\right)  }{\lambda\left(
G\right)  -\mu\left(  G\right)  }\geq1+\frac{2m/n}{\lambda\left(  G\right)
-2m/n}=1+\frac{2m}{n\lambda\left(  G\right)  -2m}.
\]

\end{document}